\documentclass[a4paper]{amsart}

\usepackage{amsmath}
\usepackage{url}
\usepackage{amsthm}
\usepackage{amsfonts}
\usepackage{mathrsfs}
\usepackage{color}
\usepackage{enumerate}
\usepackage{amssymb}
\usepackage{tikz} 
\usetikzlibrary{calc}

\newtheorem{theorem}{Theorem}[section]

\newtheorem{lemma}[theorem]{Lemma}

\theoremstyle{definition}
\newtheorem{definition}[theorem]{Definition}

\theoremstyle{remark}
\newtheorem{remark}[theorem]{Remark}

\numberwithin{equation}{section}

\newcommand{\R}{\mathbb{R}}

\newcommand{\ip}[2]{\left\langle #1, #2\right\rangle}
\newcommand{\pf}[2]{\frac{\partial #1}{\partial #2}}

\newcommand{\id}{\textup{id}}

\newcommand{\abs}[1]{\left|#1\right|}

\numberwithin{equation}{section}

\begin{document}
\title[]{Self-Expanders to Inverse Curvature Flows by Homogeneous Functions}
\author[Aaron Chow]{Tsz-Kiu Aaron Chow}
\author[Ka-Wing Chow]{Ka-Wing Chow}
\author[Frederick Fong]{Frederick Tsz-Ho Fong}
\address{Department of Mathematics, Hong Kong University of Science and Technology, Clear Water Bay, Kowloon, Hong Kong}

\email{zzhoual@connect.ust.hk}
\email{kwchowac@connect.ust.hk}
\email{frederick.fong@ust.hk}

\maketitle
\begin{abstract}
In this paper, we study self-expanding solutions to a large class of parabolic inverse curvature flows by homogeneous symmetric functions of principal curvatures in Euclidean spaces. These flows include the inverse mean curvature flow and many nonlinear flows in the literature.

We first show that the only compact self-expanders to any of these flows are round spheres. Secondly, we show that complete non-compact self-expanders to any of these flows with asymptotically cylindrical ends must be rotationally symmetric. Thirdly, we show that when such a flow is uniformly parabolic, there exist complete rotationally symmetric self-expanders which are asymptotic to two round cylinders with different radii. These extend some earlier results in \cite{DLW, DFL, KLP} to a wider class of curvature flows.
\end{abstract}

\section{Introduction}

In this article, we study a class of parabolic inverse curvature flows on Euclidean hypersurfaces $\Sigma^{n\geq2} \subset \R^{n+1}$ of the general form:
\begin{equation}
\label{eq:FlowInverse}
\left(\pf{F}{t}\right)^\perp = -\frac{1}{\rho}\nu.	
\end{equation}
Here $\nu$ is the evolving Gauss map, and $\rho(\lambda_1,\cdots,\lambda_n) : \Gamma \subset \R^n \to \R_+$ is a positive, homogeneous of degree $1$, symmetric and $C^1$ function of principal curvatures defined on an open cone $\Gamma \subset \R^n$. By homogeneous of degree $1$ we mean for any $(\lambda_1, \ldots, \lambda_n) \in \Gamma$ and $c > 0$, we have:
\[\rho(c\lambda_1, \cdots, c\lambda_n) = c\rho(\lambda_1, \cdots, \lambda_n).\]
We will call \eqref{eq:FlowInverse} the \emph{inverse $\rho$-flow}. In order for the flow \eqref{eq:FlowInverse} to have short-time existence, we require $\rho$ to satisfy parabolic conditions. We say the inverse $\rho$-flow is \emph{parabolic} if:
\[\pf{\rho}{\lambda_i} > 0 \quad \text{ on } \Gamma\]
for all $i = 1, \cdots, n$. It is a large class of flows containing many well-known flows in the literature, including:
\begin{itemize}
\item $\rho = H = \sum_i \lambda_i$: It is the well-known \emph{inverse mean curvature flow}. Here the function $\rho$ is defined on the mean-convex cone $\Gamma = \{\lambda_1 + \cdots + \lambda_n > 0\}$. It is parabolic since $\pf{\rho}{\lambda_i} = 1$.
\item $\rho = \frac{\sigma_{k+1}}{\sigma_k}$: Here $\sigma_k$ denotes the $k$-th symmetric polynomial of principal curvatures $\sum_{1 \leq i_1 < \cdots < i_k \leq n} \lambda_{i_1} \cdots \lambda_{i_k}$. It is called the \emph{inverse $Q_k$-flow} by some authors. Here $\rho$ is positive on the cone $\Gamma = \{\sigma_{k+1} > 0\}$. The flow is parabolic on $\Gamma$ by the Newton-Maclaurin's inequality.
\item $\rho = \left(\frac{\sigma_j}{\sigma_i}\right)^{\frac{1}{j-i}}$, where $j > i$, defined on the cone $\Gamma = \{\sigma_j > 0\}$. It is parabolic on $\Gamma$ again by the Newton-Maclaurin's inequality.
\item $\rho = \sigma_k^{1/k}$, where $1 \leq k \leq n$, defined on the cone 
\[\Gamma = \{\sigma_k > 0\} \cap \underbrace{\left(\bigcap_{i=1}^n \{\sigma_{k-1}(\lambda_1,\cdots,\hat{\lambda}_i,\cdots,\lambda_n) > 0\}\right)}_{\text{to guarantee }\pf{\rho}{\lambda_i} > 0 \text{ for all } i}.\]
Here $(\lambda_1, \cdots, \hat{\lambda}_i, \cdots, \lambda_n)$ means $(\lambda_1, \cdots, \lambda_n)$ with $\hat{\lambda}_i$ removed.
\item $\rho = \left(\sum_{i=1}\lambda_i^p\right)^{1/p}$ where $p > 1$, defined on the convex cone $\Gamma = \{\lambda_i > 0 \text{ for any } i\}$. It is parabolic on $\Gamma$ by direct computation of $\pf{\rho}{\lambda_i}$.
\item $\rho(\lambda_1, \lambda_2) = \frac{(\lambda_1+\lambda_2)^3}{(\lambda_1-\lambda_2)^2}$: it is a less trivial example which is positive and satisfies $\partial_i \rho > 0$ for any $i = 1, 2$ on the cone
\[\Gamma = \{(\lambda_1, \lambda_2) : \lambda_1 + \lambda_2 > 0 \text{ and } (\lambda_1 - 5\lambda_2 > 0 \text{ or } 5\lambda_1 - \lambda_2 < 0)\}.\]
This $\rho$ function is well-defined for surfaces close to round cylinders, but is undefined for the round sphere. Note that unlike the previous examples, the cone $\Gamma$ on which $\rho$ is parabolic does not contain the whole convex cone $\{\lambda_i > 0 \text{ for any } i\}$.
\end{itemize}

The study of inverse curvature flows dates back to the works by Gerhardt \cite{G} and Urbas \cite{U}, who considered parabolic flows by curvature functions $\rho$'s which are weakly concave, and independently proved that if a compact (without boundary) hypersurface $\Sigma^n$ is initially star-shaped, then the flow evolves the hypersurface into a round sphere after rescaling.

In the case $\rho = H$, the flow \eqref{eq:FlowInverse} becomes the well-known inverse mean curvature flow. This flow appears in various contexts ranging from geometric inequalities to general relativity. In \cite{HI01}, Huisken and Ilmanen used the monotonicity of the Hawking mass along the inverse mean curvature flow to prove the Riemannian Penrose inequality. Many monotone quantities along various inverse curvature flows were observed in the literature including Bray \cite{B}, Brendle-Hung-Wang \cite{BHW}, Guan-Li \cite{GL}, Kwong-Miao \cite{KM14,KM15}, Li-Wei \cite{LW}, Streets \cite{S}  (this list is far from being complete), and many useful geometric inequalities are derived using these monotone quantities.

\emph{Self-expanders} to the inverse $\rho$-flow \eqref{eq:FlowInverse} are defined by the equation:
\begin{equation}
\label{eq:SelfSimilarInverse}
-\frac{1}{\rho} = \mu\,\langle F, \nu\rangle.
\end{equation}
where $\mu > 0$. If $F$ satisfies \eqref{eq:SelfSimilarInverse}, it can be easily verified, using the homogeneity of $\rho$, that $F_t := e^{\mu t}F$ is the solution of the inverse $\rho$-flow starting from $F$. They are sometimes stationary along some monotone quantities mentioned above, and can be regarded as fixed points of the flow modulo rescaling. Therefore, it is an interesting question to search for examples of self-expanders, and study their rigidity and uniqueness problems. 

Along the inverse $\rho$-flow, round spheres (provided that $(1,\cdots,1) \in \Gamma$), and round cylinders (provided that $(1,\cdots,1,0) \in \Gamma$) are examples of self-expanders. For the inverse mean curvature flow ($\rho = H$), it was shown by Drugan, Lee, and Wheeler in \cite{DLW} that the only compact self-expanders are round spheres. Their proof used Hsiung-Minkowski's integral formulas on symmetric polynomials $\sigma_k$'s. It was subsequently generalized by Kwong, Lee, and Pyo in \cite{KLP} to a wider class of flows, including the inverse $Q_k$-flow, again using Hsiung-Minkowski's integral formulas. The first main result (Theorem \ref{thm:UniquenessCompact}) in this article is to prove that round spheres are the only compact self-expanders to general parabolic inverse $\rho$-flows, extending results in \cite{DLW, KLP} to a much wider class of flows. Our approach is different from those in \cite{DLW,KLP}. Instead, we will show that any such compact self-expander must be rotationally symmetric about \emph{any} axis passing through the origin, using similar ideas appeared in \cite{DFL} by Drugan, the third-named author, and Lee on asymptotically cylindrical inverse mean curvature self-expanders.

Since we have uniqueness for compact self-expanders, it is therefore natural to ask whether round cylinders are the only complete non-compact self-expanders. For the inverse mean curvature flow, it has already been known to be not true. There are a number of one-ended complete self-expanders to the inverse mean curvature flow constructed by Huisken-Ilmanen \cite{HI97} using phase-portrait analysis. Recently, there are two-ended complete self-expanders to the same flow constructed by Drugan, Lee, and Wheeler in \cite{DLW}, which resemble the shape of infinite wine bottles. While there is no hope of proving uniqueness of non-compact self-expanders to the inverse mean curvature flow, the rotational symmetry of such a self-expander was proved to be very rigid. It was proved in \cite{DFL} by Drugan, the third-named author, and Lee that if a complete non-compact self-expander to the inverse mean curvature flow is asymptotically cylindrical, then it must be rotationally symmetric about the axis of the cylinder. Our second main result is to extend these rotational rigidity results to general parabolic inverse $\rho$-flows (Theorem \ref{thm:RigidityNoncompact}).

Our third main result (Theorem \ref{thm:InfiniteBottle}) constructs complete non-compact self-expanders to \emph{uniformly} parabolic inverse $\rho$-flows. Their shapes resemble the \emph{infinite bottle} self-expanders to the inverse mean curvature flow constructed in \cite{DLW}. They are rotationally symmetric, two-ended, and are topologically equivalent to cylinders. However, they are \emph{geometrically} different from round cylinders, but rather are bounded between two cylinders of different radii (see Figure \ref{fig:InfiniteBottle}). This proves that, even within the same topological type, complete non-compact self-expanders to any of these uniformly parabolic inverse $\rho$-flows are not unique.

Here we state our results precisely:
\begin{theorem}[Main Theorem]
Consider the inverse $\rho$-flow \eqref{eq:FlowInverse} by a degree $1$ homogeneous symmetric $C^1$ function of principal curvatures $\rho : \Gamma \to \R_+$ defined on an open cone $\Gamma \subset \R^n$ and $\frac{\partial\rho}{\partial\lambda_i} > 0$ on $\Gamma$ for any $i$. Then, we have the following uniqueness, rigidity, and existence results:
\begin{enumerate}
\item Suppose $(1, \cdots, 1) \in \Gamma$, the only compact self-expanders to the inverse $\rho$-flow \eqref{eq:FlowInverse} must be round spheres (see Theorem \ref{thm:UniquenessCompact}).
\item Suppose $(1,\cdots,1,0) \in \Gamma$, and $\Sigma$ is a complete non-compact self-expander to the inverse $\rho$-flow \eqref{eq:FlowInverse} such that either
	\begin{itemize}
	\item[(i)] $\Sigma$ has exactly one end which is $C^2$-asymptotic to a round cylinder (in the sense of Definition \ref{def:ACEnds}), or
	\item[(ii)] $\Sigma$ has exactly two ends which are $C^2$-asymptotic to two co-axial round cylinders (in the sense of Definition \ref{def:ACEnds}), 
	\end{itemize}
then $\Sigma$ must be rotationally symmetric about the axis of the cylinder(s) (see Theorem \ref{thm:RigidityNoncompact}).
\item Suppose $\rho$ and $\Gamma$ satisfy the following conditions
\begin{itemize}
\item[(i)] $(1, \cdots, 1, \xi) \in \Gamma$ for any $\xi \in (-\alpha, \beta)$, where $\alpha > 0$ and $\beta > 1$; and
\item[(ii)] $\rho(1,\cdots,1,\xi) \to 0$ as $\xi \to -\alpha$; and
\item[(iii)] there exists a constant $C > 0$ such that $C \geq \frac{\partial\rho}{\partial\lambda_n}(1,\cdots,1,\xi) > 0$ for any $\xi \in (-\alpha,\beta)$,
\end{itemize}
then there exists a complete two-ended rotationally symmetric self-expander $\Sigma$ to the inverse $\rho$-flow \eqref{eq:FlowInverse}, which is bounded between two co-axial cylinders of different radii, and $C^2$-asymptotic to these two cylinders at both ends (see Theorem \ref{thm:InfiniteBottle}).
\end{enumerate}

\end{theorem}

\vskip 0.5cm
\noindent\textbf{Acknowledgement:} Aaron Chow is partially supported by the HKUST Undergraduate Research Opportunity Program (UROP). Ka-Wing Chow is partially supported by the HKUST postgraduate studentship. Frederick Fong is partially supported by the Hong Kong RGC Early Career Grant \#26301316 and the Hong Kong RGC General Research Grant \#16302417. He would like thank Hojoo Lee for numerous and helpful discussions on the inverse mean curvature flow, especially on his recent joint paper \cite{DLW}.

The authors would like to thank the anonymous referee for his/her very careful reading and insightful comments, and for his/her valuable suggestions to improve the exposition of the previous version. We are grateful that he/she pointed out a careless error in \eqref{eq:(r'Q)'} of the previous version and suggested a way to fix it.

\section{Preliminaries}
Let $\Sigma^n$ be a complete manifold of dimension $n \geq 2$ without boundary. Suppose $F : \Sigma^n \to \mathbb{R}^{n+1}$ is an immersion of $\Sigma$ into $\R^{n+1}$. Let $\nu$ be the Gauss map of $F$ which is chosen to be inward-pointing when $\Sigma$ is closed. Under local coordinates $(u_1, \ldots, u_n)$ of $\Sigma$, we denote
\[g_{ij} = \ip{\pf{F}{u_i}}{\pf{F}{u_j}}\]
to be the first fundamental form of $\Sigma$ and $g^{ij}$ be the matrix inverse of $[g_{ij}]$. We adopt the following convention of the second fundamental form, shape operator and mean curvature:
\begin{align*}
h_{ij} & = \ip{\frac{\partial^2 F}{\partial u_i \partial u_j}}{\nu}\\
h_i^j & = g^{jk}h_{ik}\\
H & = h_i^i = g^{ij}h_{ij}
\end{align*}
The principal curvatures $\lambda_1, \ldots, \lambda_n$ are the eigenvalues of the matrix $[h_i^j]$.

Now we consider a more general type of curvature flows by homogeneous functions of any degree. Let $\varphi : \Gamma \subset \R^n \to \R$ be symmetric function of the principal curvatures of degree $\alpha$ defined on an open cone $\Gamma$ in $\R^n$, meaning that
\[\varphi(c\lambda_1, \ldots, c\lambda_n) = c^\alpha \varphi(\lambda_1, \ldots, \lambda_n)\]
for any $(\lambda_1, \ldots, \lambda_n) \in \Gamma$ and $c > 0$. We consider the curvature flow:
\begin{equation}
\label{eq:Flow}
\left(\pf{F_t}{t}\right)^\perp = \varphi_t \nu_t
\end{equation}
with initial condition $F_0 = F$. Here $\varphi_t := \varphi(\lambda_i(t))$ where $\lambda_i(t)$'s are the principal curvatures of $F_t$, and $\nu_t := \nu(F_t)$ is the Gauss map of $F_t$.

When $\varphi = -\frac{1}{\rho}$ where $\rho$ is homogeneous of degree $1$, then \eqref{eq:Flow} is a flow by a homogeneous function of degree $-1$. Generally, the notions of self-similar solutions are defined as follows:

\begin{definition}[Self-Similar Solutions]
We say $F$ is a \emph{self-similar solution} to \eqref{eq:Flow} if there exists a constant $\mu \not= 0$ such that
\begin{equation}
\label{eq:SelfSimilar}
\varphi = \mu\ip{F}{\nu}.	
\end{equation}

Furthermore,
\begin{itemize}
\item If $\mu > 0$, we say $F$ is a \emph{self-expander} to \eqref{eq:Flow}. 
\item If $\mu < 0$, we say $F$ is a \emph{self-shrinker} to \eqref{eq:Flow}.
\end{itemize}
\end{definition}

When $\deg \varphi = \alpha \not= -1$, then $F$ satisfies \eqref{eq:SelfSimilar} if and only if
\[F_t = \left[1 + \mu(\alpha + 1)t\right]^{\frac{1}{\alpha+1}}F\]
satisfies \eqref{eq:Flow}. When $\deg \varphi = -1$, then $F$ satisfies \eqref{eq:SelfSimilar} if and only if $F_t = e^{\mu t}F$ satisfies \eqref{eq:Flow}. Therefore, the case $\deg \varphi = -1$ is somewhat special, and deserves special attentions.

When $\varphi = H$, the flow \eqref{eq:Flow} is the well-known mean curvature flow. Huisken proved in \cite{H84} that any convex closed hypersurface evolves under the flow to a round sphere. It was later generalized to flows with $\deg\varphi = 1$ by Andrews \cite{A94} (see also \cite{ALM, AMZ}). The study of self-similar solutions is of intensive interest in the literature, since it is important for understanding singularity formation (see e.g. Huisken \cite{H86}, Colding-Minicozzi \cite{CM12}, et. al). In \cite{H86}, Huisken proved that the only mean-convex compact self-shrinkers are round spheres. The result was later generalized by McCoy \cite{Mc} to a wider class of flows with $\deg \varphi = 1$. In $\R^3$ when $\Sigma^2$ has genus $0$, it was spectacularly proved by Brendle \cite{B16} that the only closed self-shrinker is the round sphere.

Non-compact self-similar solutions in the case $\deg \varphi = 1$ are also very rigid. Wang established in \cite{W14,W16} uniqueness of asymptotically conical and cylindrical mean curvature self-shrinkers (see also a recent generalization to flows with $\deg\varphi = 1$ by Guo \cite{Gu}). Some rotationally symmetric examples of asymptotically conical self-shrinkers and self-expanders are constructed by, for instance, Kleene-M\o ller \cite{KMo14}, Angenent-Chopp-Ilmanen \cite{ACI}, Helmansdorffer \cite{Hel}, Guo \cite{G}. Rotational rigidity of asymptotically conical mean curvature self-expanders was established in \cite{FM} by the third-named author and McGrath.

Similar to works about proving rotational symmetry of self-similar solutions, such as \cite{B13,B14,C,CF,Hf,BL,DFL,FM}, we will derive some relevant Jacobi-type equations and apply the maximum principle on them. We need the following well-known variational formulae, whose proofs can be found in, for instance, Andrews' \cite{A94} and Huisken-Polden's \cite{HP} works.

\begin{lemma}[c.f. \cite{A94}, \cite{HP}]
\label{lma:Variations}
Suppose $F_s : \Sigma^n \to \R^{n+1}$ is a smooth family of hypersurfaces satisfying $\displaystyle{\frac{\partial F_s}{\partial s} = f_s \nu_s}$, where each $f_s : \Sigma \to \R$ is a smooth scalar function, and $\nu_s$ is the Gauss map of $F_s$. Let $\varphi(\{h_j^i\})$ be a scalar function of the shape operator. Then, we have the following evolution equations:
\begin{align}
\label{eq:g} \frac{\partial}{\partial s} g_{ij} & = -2fh_{ij}\\
\label{eq:nu} \frac{\partial \nu}{\partial s} & = -\nabla f\\
\label{eq:h} \frac{\partial}{\partial s}h_{ij} & = \nabla_i \nabla_j f - fh_{ik}h_j^k\\
\label{eq:varphi} \frac{\partial \varphi}{\partial s} & =  g^{il}\pf{\varphi}{h_j^i}\nabla_l \nabla_j f + f h_p^i h_j^p \pf{\varphi}{h_j^i},
\end{align}
where all geometric quantities above are with respect to $F_s$.
\end{lemma}

\begin{proof}
Detail computations can be found in e.g. \cite[Theorem 3.7]{A94} or \cite[Lemma 7.6]{HP}. For readers' convenience, we provide the proof of \eqref{eq:varphi}. Using \eqref{eq:g} and \eqref{eq:h}, we have
\begin{align*}
\frac{\partial \varphi}{\partial s} &= \frac{\partial \varphi}{\partial h_j^i} \frac{\partial}{\partial s}(g^{il}h_{lj}) \\
&=  \frac{\partial \varphi}{\partial h_j^i}\big(2fg^{ip}g^{ln}h_{pn}h_{lj} + g^{il}\nabla_l\nabla_j f -fg^{il}h_{lp}h_j^p\big) \\
&=   g^{il}\frac{\partial \varphi}{\partial h_j^i}\nabla_l\nabla_j f + fh_p^i h_j^p  \frac{\partial \varphi}{\partial h_j^i}
\end{align*}
as desired.
\end{proof}

\section{Eigenfunctions of Stability Operator}
Next we derive several Jacobi-type equations of various geometric quantities that measure rotational or translational invariances for the self-expanders. The approach of proving them is similar to that in \cite{DFL,FM}. The major difference is that the Laplacian operators $\Delta$ in \cite{DFL,FM} are replaced by a more general elliptic operator $g^{il}\pf{\varphi}{h_j^i}\nabla_l \nabla_j$.
\begin{lemma}
\label{lma:L(varphi)}
Let $F : \Sigma^n \to \R^{n+1}$ be a self-similar solution to \eqref{eq:Flow}. Then, $\varphi$ satisfies the following second-order equation:
\begin{equation}
\label{eq:L(varphi)}
g^{il}\pf{\varphi}{h_j^i}\nabla_l \nabla_j \varphi + \mu\ip{F}{\nabla\varphi} + \left(\pf{\varphi}{h_j^i}h_p^i h_j^p - \mu\right)\varphi = -(\alpha + 1)\mu\varphi.
\end{equation}
\end{lemma}
\begin{proof}
Recall that for a self-similar solution $F$ satisfying \eqref{eq:SelfSimilar}, the solution to the flow \eqref{eq:Flow} starting from $F$ is given by $F_t = \psi_t F$ where
\[\psi_t :=
\begin{cases}
e^{\mu t} & \text{ if } \alpha = -1\\
\left[1 + \mu(\alpha + 1)t\right]^{\frac{1}{\alpha+1}} & \text{ if } \alpha \not= -1
\end{cases}.
\]
To prove \eqref{eq:L(varphi)}, we consider a local reparametrization $\Phi_t$ with $\Phi_0 = \id$ such that the flow \eqref{eq:Flow} becomes a normal flow:
\begin{equation}
\label{eq:FlowNormal}
\frac{\partial (F_t \circ \Phi_t)}{\partial t} = \varphi (F_t \circ \Phi_t) \nu(F_t \circ \Phi_t)
\end{equation}
Applying \eqref{eq:varphi} in Lemma \ref{lma:Variations} with $f_t := F_t \circ \Phi_t$, we get 
\begin{equation}
\label{eq:L(varphi)_RHS}
\frac{\partial \varphi (F_t \circ \Phi_t)}{\partial t} \bigg|_{t=0} =  g^{il}\frac{\partial \varphi}{\partial h_j^i}\nabla_l\nabla_j \varphi + \varphi h_p^i h_j^p  \frac{\partial \varphi}{\partial h_j^i}
\end{equation}
On the other hand, since $\varphi$ is homogeneous of degree $\alpha$ and the self-similar solution equation \eqref{eq:SelfSimilar} is invariant under reparametrization, we have
\[\varphi(F_t \circ \Phi_t) = \psi_t^{-\alpha} \varphi(F \circ \Phi_t) = \psi_t^{-\alpha}\mu \,\langle F, \nu \rangle \circ \Phi_t = \psi_t^{-\alpha-1} \mu\langle F_t \circ \Phi_t, \nu(F_t \circ \Phi_t)\rangle.\]
Then, differentiating both sides using the evolution formula \eqref{eq:nu}, we can easily obtain that in either case of $\alpha = -1$ and $\alpha \not= -1$:\\
\begin{equation}
\label{eq:L(varphi)_LHS}
\frac{\partial \varphi (F_t \circ \Phi_t)}{\partial t} \bigg|_{t=0} = -\mu(\alpha +1)\varphi + \mu\varphi - \mu\langle F, \nabla \varphi \rangle.
\end{equation}
Here we have used the elementary fact that $\psi_t'(0) = \mu$ in both cases of $\alpha = -1$ and $\alpha \not= -1$. Upon combining \eqref{eq:L(varphi)_RHS} and \eqref{eq:L(varphi)_LHS}, we get the desired result \eqref{eq:L(varphi)}.\\
\end{proof}

From now on, given a rotation vector field $R$ in $\R^{n+1}$, we define $f_R : \Sigma \to \R$ by:
\[f_R := \ip{R(F)}{\nu(F)}.\]
If $\Sigma$ is rotationally invariant along $R$, then $R$ is tangential and as such $f_R \equiv 0$. Similar to \cite{DFL,FM}, we will show using elliptic method that $f_R \equiv 0$ for any relevant rotation vector field $R$, hence establishing rotational symmetry. For this, we need the following Jacobi-type equation:
\begin{lemma}
\label{lma:L(f_R)}
Let $F : \Sigma^n \to \R^{n+1}$ be a self-similar solution to \eqref{eq:Flow} and $R$ be a rotation Killing vector field in $\R^{n+1}$. Then, $f_R$ satisfies the following second-order equation:
\begin{equation}
\label{eq:L(f_R)}
g^{il}\pf{\varphi}{h_j^i}\nabla_l \nabla_j f_R + \mu\ip{F}{\nabla f_R} + \left(\pf{\varphi}{h_j^i}h_p^i h_j^p - \mu\right)f_R = 0.
\end{equation}
\end{lemma}
\begin{proof}
Let $\Psi_s$ be the 1-parameter family of diffeomorphisms generated by the vector field $R$ satisfying $\frac{\partial}{\partial s}\Psi_s = R \circ \Psi_s $ with $\Psi_0 = \id$. Then, the composition $\Psi_s \circ F$ satisfies $ \frac{\partial}{\partial s} (\Psi_s \circ F) = R(\Psi_s \circ F) $. We can find a suitable reparametrization $\Phi_s$ with $\Phi_0 = \id$ such that\\
\begin{equation}
\frac{\partial}{\partial s} (\Psi_s \circ F \circ \Phi_s) = R(\Psi_s \circ F \circ \Phi_s)^{\bot} = f_R(\Psi_s \circ F \circ \Phi_s)\nu(\Psi_s \circ F \circ \Phi_s)
\end{equation}
Then, applying \eqref{eq:varphi} with $f_s : = f_R(\Psi_s \circ F \circ \Phi_s)$ yields 
\begin{equation}
\label{eq:L(f_R)_RHS}
\frac{\partial \varphi (\Psi_s \circ F \circ \Phi_s)}{\partial s} \bigg|_{s=0} =  g^{il}\frac{\partial \varphi}{\partial h_j^i}\nabla_l\nabla_j f_R + f_R h_p^i h_j^p  \frac{\partial \varphi}{\partial h_j^i}
\end{equation}
On the other hand, principal curvatures (and hence $\varphi$) are invariant under rotations, thus from the self-similar solution equation \eqref{eq:SelfSimilar}, we get
\begin{align*}
\varphi(\Psi_s \circ F \circ \Phi_s) = \mu \langle \Psi_s \circ F \circ \Phi_s, \nu(\Psi_s \circ F \circ \Phi_s) \rangle
\end{align*}
Differentiating both sides, we get
\begin{equation}
\label{eq:L(f_R)_LHS}
\frac{\partial \varphi(\Psi_s \circ F \circ \Phi_s)}{\partial s} \bigg|_{s=0}= \mu f_R - \mu \langle F, \nabla f_R \rangle.
\end{equation}
Here we have used \eqref{eq:nu} with $f_s = \Psi_s \circ F \circ \Phi_s$. Upon combining \eqref{eq:L(f_R)_RHS} and \eqref{eq:L(f_R)_LHS}, we get the desired result \eqref{eq:L(f_R)}.\\
\end{proof}

\begin{lemma}
Let $F : \Sigma \to \R^{n+1}$ be a self-similar solution to \eqref{eq:Flow}. Consider a constant vector field $V$ in $\R^{n+1}$. Then, the function $\ip{V}{\nu}$ satisfies the following second-order equation:
\begin{equation}
\label{eq:L(V)}
g^{il}\pf{\varphi}{h_j^i}\nabla_l \nabla_j \ip{V}{\nu} + \mu\ip{F}{\nabla\ip{V}{\nu}} + \left(\pf{\varphi}{h_j^i}h_p^i h_j^p - \mu\right)\ip{V}{\nu} = -\mu\ip{V}{\nu}.
\end{equation}
\end{lemma}
\begin{proof}
The proof is similar to previous two lemmas. The key idea is to consider the family of translating hypersurfaces defined by $F_{\tau} := F+\tau V$ so that 
\[\left(\frac{\partial F_{\tau}}{\partial \tau}\right)^{\perp} = \langle V, \nu \rangle \nu.\] From the fact that principal curvatures are invariant under translations, we have
\begin{align*}
\varphi(F_{\tau}) =\varphi(F) = \mu \langle F, \nu(F) \rangle =\mu \langle F, \nu(F_{\tau}) \rangle
\end{align*}
Then \eqref{eq:L(V)} follows from differentiating both sides at $\tau = 0$ by applying \eqref{eq:nu} and \eqref{eq:varphi}.\\
\end{proof}

For each $\varphi$ under consideration, we define a second-order operator $L_\varphi : C^2(M) \to C^0(M)$ by:
\begin{equation}
L_\varphi f := g^{il}\pf{\varphi}{h_j^i}\nabla_l \nabla_j f + \mu\ip{F}{\nabla f} + \left(\pf{\varphi}{h_j^i}h_p^i h_j^p - \mu\right)f
\end{equation}
Note that $L_\varphi$ is not necessarily elliptic. One sufficient condition to guarantee the ellipticity of $L_\varphi$ is that the matrix $\displaystyle{\left[\pf{\varphi}{h_i^j}\right]}$
is positive-definite, or equivalently, $\displaystyle{\pf{\varphi}{\lambda_i} > 0}$ for any $i = 1, \cdots, n$.

Using this operator $L_\varphi$, one can express \eqref{eq:L(varphi)}, \eqref{eq:L(f_R)} and \eqref{eq:L(V)} as:
\begin{align*}
L_\varphi\varphi & = -(\alpha + 1)\mu\varphi\\
L_\varphi f_R & = 0\\
L_\varphi	 \ip{V}{\nu} & = -\mu\ip{V}{\nu}
\end{align*}
meaning that $\varphi$, $f_R$ and $\ip{V}{\nu}$ are eigenfunctions of $L_\varphi$. In particular, when $\deg \varphi = -1$, we have $L_\varphi \varphi = 0$ so that both $\varphi$ and $f_R$ are both in the kernel of $L_\varphi$. It again makes the case $\deg \varphi = -1$ very special from other degrees.

The following lemma will be used to make comparison between two eigenfunctions of $L_\varphi$ with different eigenvalues.
\begin{lemma}
\label{lma:Quotient}
Let $F : \Sigma \to \R^{n+1}$ be a self-similar solution to \eqref{eq:Flow}. Then, for any $\lambda_1, \lambda_2 \in \mathbb{R}$ and $f_1,f_2 \in C^2(\Sigma)$ where $f_2 \neq 0$ on $\Sigma$ such that $L_{\varphi} f_1 = \lambda_1 f_1$ and $L_{\varphi} f_2 = \lambda_2 f_2$, we have
\begin{align}
\label{eq:Quotient}
& g^{il}\frac{\partial \varphi}{\partial h_j^i}\nabla_l \nabla_j \left(\frac{f_1}{f_2}\right)\\
\nonumber & =  -\mu \left\langle F, \nabla\left(\frac{f_1}{f_2}\right) \right\rangle + \frac{f_1}{f_2}(\lambda_1 - \lambda_2) - \frac{2}{f_2}g^{il}\frac{\partial \varphi}{\partial h_j^i}\nabla_l \left(\frac{f_1}{f_2}\right) \nabla_j f_2
\end{align}
\end{lemma}
\begin{proof}
From $L_{\varphi} f_1 = \lambda_1 f_1$ and $L_{\varphi} f_2 = \lambda_2 f_2$, we have \\
\begin{equation}
\begin{aligned}
\label{eq:Eigens}
g^{il}\frac{\partial \varphi}{\partial h_j^i} \nabla_l \nabla_j f_1 &= -\mu \langle F, \nabla f_1 \rangle - \left(\frac{\partial \varphi}{\partial h_j^i}h_p^i h_j^p - \lambda_1+\mu\right)f_1 \\
g^{il}\frac{\partial \varphi}{\partial h_j^i} \nabla_l \nabla_j f_2 &= -\mu \langle F, \nabla f_2 \rangle - \left(\frac{\partial \varphi}{\partial h_j^i}h_p^i h_j^p - \lambda_2+\mu\right)f_2
\end{aligned}
\end{equation} \\
Using \eqref{eq:Eigens}, and the well-known quotient identity\\
\begin{align*}
& \nabla_i \nabla_j \left(\frac{h}{k}\right)\\
& = \frac{k \nabla_i \nabla_j h - h\nabla_i \nabla_j k}{k^2} - \frac{2}{k}\nabla_i\left(\frac{h}{k}\right)\nabla_j h
\end{align*}
which holds for any $h,k \in C^2(\Sigma)$ wherever $k \not= 0$, we can derive using \eqref{eq:Eigens} that:
\begin{align*}
& g^{il}\frac{\partial \varphi}{\partial h_j^i}\nabla_l\nabla_j \left(\frac{f_1}{f_2}\right)\\
& = \frac{f_2\left[-\mu \langle F, \nabla f_1 \rangle - \left(\dfrac{\partial \varphi}{\partial h_j^i}h_p^i h_j^p - \lambda_1+\mu\right)f_1\right]}{f_2^2}\\
& \quad -\frac{f_1 \left[-\mu \langle F, \nabla f_2 \rangle - \left(\dfrac{\partial \varphi}{\partial h_j^i}h_p^i h_j^p - \lambda_2+\mu\right)f_2\right]}{f_2^2} \\
& \quad -\frac{2}{f_2}g^{il}\frac{\partial \varphi}{\partial h_j^i}\nabla_l \left(\frac{f_1}{f_2}\right) \nabla_j f_2 \\
&= -\mu \left\langle F, \frac{f_2 \nabla f_1 - f_1 \nabla f_2}{f_2^{2}} \right\rangle + \frac{f_1}{f_2}(\lambda_1 - \lambda_2) - \frac{2}{f_2}g^{il}\frac{\partial \varphi}{\partial h_j^i}\nabla_l \left(\frac{f_1}{f_2}\right) \nabla_j f_2 \\
&=  -\mu \left\langle F, \nabla\left(\frac{f_1}{f_2}\right)\right\rangle + \frac{f_1}{f_2}(\lambda_1 - \lambda_2) - \frac{2}{f_2}g^{il}\frac{\partial \varphi}{\partial h_j^i}\nabla_l \left(\frac{f_1}{f_2}\right) \nabla_j f_2,
\end{align*}
as desired.
\end{proof}

\section{Uniqueness of Compact Self-Expanders}
In the rest of this article, we will focus on inverse curvature flows \eqref{eq:FlowInverse} by degree $-1$ homogeneous symmetric function $\rho$'s of principal curvatures as described in the introduction. Set $\varphi := -\frac{1}{\rho}$. It is easy to see that whenever $\rho \not= 0$, we have:
\[\pf{\rho}{\lambda_i} > 0 \quad \text{if and only if} \quad \pf{\varphi}{\lambda_i} > 0.\]

Our first result concerns about compact self-expanders to the inverse $\rho$-flows which are parabolic on $\Gamma$. We are going to show that round spheres are the only compact self-expanders, thus generalizing the uniqueness result for compact self-expanders to the inverse mean curvature flow proved by Drugan, Lee, and Wheeler in \cite{DLW}, and for inverse $Q_k$-flows proved by Kwong, Lee, and Pyo in \cite{KLP}.

In both \cite{DLW} and \cite{KLP}, the authors obtained the uniqueness results using Hsiung-Minkowski's integral formulae, which concern about \emph{symmetric polynomials} $\sigma_k$'s of principal curvatures. In order to extend their results to a more general class of flows where $\rho$ are not explicit, we use arguments different from those in \cite{DLW} and \cite{KLP}. Applying the ideas in the work \cite{DFL} of Drugan, the third-named author, and Lee about rotational rigidity of self-expanders to inverse mean curvature flow, we will show that compact self-expanders of such a general inverse $\rho$-flow must be rotationally symmetric along the rotation vector field $R$ about any axis through the origin, hence they must be round spheres centered at the origin.

The principal curvatures of the round sphere of radius $r_0$ are $(r_0^{-1}, \cdots, r_0^{-1})$  and $\langle F, \nu\rangle = - r_0$ when $\nu$ is taken to inward-pointing. Therefore, it is a self-expander to the inverse $\rho$-flow whenever $\rho(1,\cdots,1) > 0$, and the constant $\mu$ is given by:
\[-\frac{1}{\rho(r_0^{-1},\ldots,r_0^{-1})} = -\mu r_0 \quad \Longrightarrow \quad \mu = \frac{1}{\rho(1,\cdots,1)} > 0.\]
We will only consider function $\rho$'s with property that $\rho(1,\cdots,1) > 0$. Let's state our uniqueness theorem precisely:

\begin{theorem}
\label{thm:UniquenessCompact}
Consider the inverse $\rho$-flow by a degree $1$ homogeneous symmetric $C^1$ function of principal curvatures $\rho(\lambda_1, \cdots, \lambda_n) : \Gamma \to \R_+$ such that $(1, \cdots, 1) \in \Gamma$ and $\dfrac{\partial\rho}{\partial\lambda_i} > 0$ on $\Gamma$ for any $i$. Then, the only compact self-expanders to the inverse $\rho$-flow must be round spheres.
\end{theorem}

\begin{proof}
Set $\varphi := -1/\rho$. Let $R$ be an arbitrary rotation vector field in $\R^{n+1}$ about an axis passing through the origin. We want to show $f_R = \langle R,\nu\rangle \equiv 0$ on $\Sigma$.

By Lemma \ref{lma:L(varphi)} and Lemma \ref{lma:L(f_R)}, we have
\begin{align*}
L_{\varphi} \varphi &= 0 \\
L_{\varphi} f_R &= 0
\end{align*}
Then, we apply Lemma \ref{lma:Quotient} to $f_1 = f_R$ and $f_2 = \varphi$ and get \\
\begin{equation}
\label{eq:EllipticCompact}
g^{il}\frac{\partial \varphi}{\partial h_j^i} \nabla_l \nabla_j \left(\frac{f_R}{\varphi}\right) =  -\mu \left\langle F, \nabla\left(\frac{f_R}{\varphi}\right)\right\rangle  + \frac{2}{\varphi}g^{il}\frac{\partial \varphi}{\partial h_j^i}\nabla_l \left(\frac{f_R}{\varphi}\right) \nabla_j \varphi
\end{equation}
Note that \eqref{eq:EllipticCompact} is a second-order elliptic equation by the parabolic condition on $\rho$. Furthermore, by the compactness of $\Sigma$, the set of all possible principal curvatures
\[\{(\lambda_1(p), \ldots, \lambda_n(p)) : p \in \Sigma\}\]
is compact in $\Gamma$. As a result, the equation \eqref{eq:EllipticCompact} is uniformly elliptic.

By the compactness of $\Sigma$, the function $\frac{f_R}{\varphi}$ achieves a global maximum. Observing that the RHS of \eqref{eq:EllipticCompact} involves only gradient terms of $\frac{f_R}{\varphi}$, we can apply the elliptic strong maximum principle and conclude that $\frac{f_R}{\varphi} \equiv C$, where $C$ is a constant. Moreover, as the rotational axis of $R$ passes through the origin, we have $\textup{div}\,(R) = 0$, and so Divergence Theorem shows that \\
\begin{align*}
\oint_{\Sigma} f_R\, d\Sigma = \oint_{\Sigma} \langle R, \nu \rangle\,d\Sigma = 0
\end{align*}\\
Hence we have $f_R(p) = 0$ for some $p \in \Sigma$, and so the constant $C$, and hence $f_R$, must be $0$. Therefore, $\Sigma$ is rotationally symmetric about $R$. Since $R$ can be any rotational vector field with axis through the origin, we conclude that $\Sigma$ must be a round sphere.
\end{proof}

\section{Asymptotically Cylindrical Self-Expanders}
Our second result concerns about non-compact self-expanders to parabolic inverse $\rho$-flows. Again we take $\nu$ to be inward-pointing unit normal for the cylinder. A round cylinder of radius $r_0$ has principal curvatures $(r_0^{-1}, \cdots, r_0^{-1},0)$, and $\langle F, \nu\rangle = -r_0$. Therefore, it is a self-expander whenever $\rho(1,\cdots,1,0) > 0$, and the constant $\mu$ is given by $\mu = \rho(1,\cdots,1,0)^{-1} > 0$.

In this section, we will extend the rotational rigidity of asymptotically cylindrical self-expanders to a general class of parabolic inverse $\rho$-flows. We first define the meaning of asymptotically cylindrical ends.

\begin{definition}[Asymptotically Cylindrical Ends]
\label{def:ACEnds}
Consider an end $E$ of a hypersurface $\Sigma^n$ is \emph{$C^k$-asymptotic} to a round cylinder $\mathcal{C}$ if there exist $r > 0$ and $u \in C^k(\mathcal{C} \backslash B_r(0))$ such that $E$ can be parametrized as a normal graph over $\mathcal{C}$ of the form:
\[F(p) = p + u(p)\nu_{\textup{cyl}}(p)\]
where $p \in \mathcal{C}$ and $\nu_{\textup{cyl}}$ is the inward-pointing unit normal of $\mathcal{C}$, and as $\abs{p} \to \infty$, the height function $u$ satisfies:
\[\abs{\nabla_{\textup{cyl}}^{(j)} u(p)}_{g_{\textup{cyl}}} \leq o(\abs{p}^{-j}) \quad \text{ for } j = 0, 1, 2, \cdots, k,\]
where $g_{\textup{cyl}}$ is the first fundamental form of $\mathcal{C}$, and $\nabla_{\textup{cyl}}$ is the covariant derivative of $\mathcal{C}$.
\end{definition}

\begin{remark}
We only require the \emph{end} $E$ to be graphical over a cylinder, but the compact part $\Sigma \backslash E$ can be of any topological type.
\end{remark}

By straight-forward computations as in \cite[Section 3]{DFL} and \cite[Section 3]{FM}, a $C^1$-asymptotically cylindrical end $E$ satisfies the following asymptotics:
\begin{itemize}
\item $F_E = F_{\textup{cyl}} + o(1)$
\item $\nu_E = \nu_{\textup{cyl}} + o(\abs{p}^{-1})$
\item $\langle F_E, \nu_E\rangle = \underbrace{\langle F_{\textup{cyl}}, \nu_{\textup{cyl}}\rangle}_{-\textup{radius of }\mathcal{C}} + o(1)$
\item $\langle R(F_E), \nu_E\rangle = \underbrace{\langle R(F_{\textup{cyl}}), \nu_{\textup{cyl}}\rangle}_{=0} + o(1) = o(1)$
\end{itemize}
Furthermore, if the end is $C^2$-asymptotically cylindrical, we further have:
\begin{itemize}
\item $\|h_E - h_{\textup{cyl}}\|_{g_{\textup{cyl}}} = o(\abs{p}^{-2})$
\item $H_E = H_{\textup{cyl}} + o(\abs{p}^{-2})$
\end{itemize}

Now we state and prove the second main result, about the rotational rigidity of asymptotically cylindrical self-expanders to any inverse $\rho$-flow, hence extending the result in \cite{DFL} to a much wider class of flows.

\begin{theorem}
\label{thm:RigidityNoncompact}
Consider the inverse $\rho$-flow by a degree $1$ homogeneous symmetric $C^1$ function of principal curvatures $\rho(\lambda_1, \cdots, \lambda_n) : \Gamma \to \R_+$ such that $(1, \cdots, 1, 0) \in \Gamma$ and $\dfrac{\partial\rho}{\partial\lambda_i} > 0$ on $\Gamma$ for any $i$. Suppose $\Sigma$ is a self-expander to the inverse $\rho$-flow such that either:
\begin{itemize}
\item $\Sigma$ has exactly one end, which is $C^2$-asymptotic to a round cylinder, or
\item $\Sigma$ has exactly two ends, which are $C^2$-asymptotic to two co-axial round cylinders,
\end{itemize}
then $\Sigma$ must be rotationally symmetric about the axis of the cylinder(s).
\end{theorem}

\begin{proof}
Similar to the proof in Theorem \ref{thm:UniquenessCompact}, we consider the elliptic equations derived in Lemma \ref{lma:L(varphi)} and Lemma \ref{lma:L(f_R)}:
\begin{align*}
L_{\varphi} f_R &= 0 \\
L_{\varphi} \langle F, \nu \rangle &= 0
\end{align*}
Recall that $\varphi = \mu\langle F,\nu\rangle$. Then, we apply Lemma \ref{lma:Quotient} on $f_1 = f_R$ and $f_2 = \langle F, \nu \rangle$:\\
\begin{equation}
\label{eq:EllipticNonCompact}
g^{il}\frac{\partial \varphi}{\partial h_j^i} \nabla_l \nabla_j \frac{f_R}{\langle F, \nu \rangle} =  -\mu \left\langle F, \nabla \frac{f_R}{\langle F, \nu \rangle} \right\rangle  - \frac{2}{\langle F, \nu \rangle}g^{il}\frac{\partial \varphi}{\partial h_j^i}\nabla_l \frac{f_R}{\langle F, \nu \rangle} \nabla_j \langle F, \nu \rangle
\end{equation}
We pick an arbitrary rotation vector field $R$ about the axis of the asymptotic cylinder. Since $\Sigma$ is $C^2$ asymptotic to cylinders, the principal curvatures $(\lambda_1, \ldots, \lambda_n) \to (r^{-1}, \cdots, r^{-1}, 0)$ uniformly when approaching to the asymptotic cylinder with radius $r$, the set of all possible principal curvatures
\[\{(\lambda_1(p), \ldots, \lambda_n(p)) : p \in \Sigma\}\]
is compact in $\Gamma$. As a result, the equation \eqref{eq:EllipticNonCompact} is uniformly elliptic.

Furthermore, we have $f_R \to 0$ uniformly at infinity, and $\abs{\langle F, \nu\rangle}$ converges uniformly to $r_0$, the radius of the asymptotic cylinder. Hence, for any $\varepsilon >0$, there is a large $s = s(\varepsilon) > 0$ so that 
\begin{equation}
\label{eq:Ends}
\sup_{\Sigma \backslash B_s(0)} \abs{\frac{f_R}{\langle F, \nu \rangle}} < \varepsilon.
\end{equation}

Now suppose $f_R \not\equiv 0$, and there exists $q \in \Sigma$ such that $f_R(q) > 0$ (the case $f_R(q) < 0$ is similar). Take $\varepsilon := \frac{f_R}{2\langle F,\nu\rangle}(q) > 0$ and let $s=s(\varepsilon)$ be sufficiently large so that $q \in B_s(0)$ and \eqref{eq:Ends} holds. Then, we have 
\[\frac{f_R}{\langle F, \nu\rangle}(q) > \varepsilon > \sup_{\Sigma \backslash B_s(0)} \abs{\frac{f_R}{\langle F, \nu \rangle}} \geq \sup_{\Sigma \cap \partial B_s(0)} \abs{\frac{f_R}{\langle F, \nu \rangle}}.\]
In other words, the function $\frac{f_R}{\langle F, \nu \rangle}$ achieves an interior maximum on $\Sigma \cap B_s(0)$, which violates the elliptic strong maximum principle applied to \eqref{eq:EllipticNonCompact} on $\Sigma \cap B_s(0)$. We get a contradiction, and so it is necessary that $f_R \equiv 0$ on $\Sigma$. This proves that $\Sigma$ is rotationally symmetric about the axis of the cylinder.
\end{proof}

\begin{remark}
\label{rmk:C1}
Note that $C^2$-asymptotics are essential in the proof of Theorem \ref{thm:RigidityNoncompact} because the strong maximum principle requires uniform ellipticity. However, if we further assume that there exists $C > 0$ such that $C \geq \pf{\rho}{\lambda_i} \geq \frac{1}{C}$ on $\Gamma$ for any $i$ (which is the case when $\rho = H$), then Theorem \ref{thm:RigidityNoncompact} still holds if the ends are just $C^1$-asymptotically cylindrical.
\end{remark}

\begin{remark}
In the two-end case considered in Theorem \ref{thm:RigidityNoncompact}, we do not require the two asymptotic cylinders to have the same radius, but they are required to be co-axial in order to have the same set of rotational Killing vector fields $R$. However, if the radius of the two cylinders are the same, then by the upcoming ODE analysis, such a self-expander must be identical to that cylinder.
\end{remark}

\section{Existence of Non-Cylindrical Self-Expanders}

A rotationally symmetric hypersurface $\Sigma^{n \geq 2}$ in $\R^{n+1}$ can be described by a profile curve $(r(s), h(s))$ in the $(r,h)$-plane. By rotating the profile curve along the $h$-axis, the hypersurface generated by the profile curve is parametrized by:
\begin{equation}
\label{eq:Rotational}
F(\omega, s) = r(s)\,\Phi(\omega) + h(s) e_{n+1},
\end{equation}
where $\Phi(\omega)$ parametrizes the unit sphere $\mathbb{S}^{n-1} \subset \R^n := \textup{span}\{e_i\}_{i=1}^n$ by the Euler's angles $\omega = (\theta_1, \cdots, \theta_{n-1})$.

For simplicity, we will only consider profile curves of the form $(r(h),h)$, i.e. graphs over the $h$-axis in the $(r,h)$-plane, and that $r(h) > 0$. The principal curvatures of \eqref{eq:Rotational} are:
\[\lambda_1 = \cdots = \lambda_{n-1} = \frac{1}{r}\cdot\frac{1}{(1+\dot{r}^2)^{1/2}}, \quad \lambda_n = -\frac{\ddot{r}}{(1+\dot{r}^2)^{3/2}}\]
where we denote $\dot{r} := \frac{dr}{dh}$ and $\ddot{r} := \frac{d^2r}{dh^2}$. By direct computations, we can also verify that:
\[\langle F, \nu\rangle = \frac{\dot{r}h - r}{(1+\dot{r}^2)^{1/2}}\]
where $\nu$ is again inward-pointing. Let $\rho : \Gamma \to \R_+$ be a positive, homogeneous of degree $1$, symmetric and $C^1$ function of principal curvatures defined on an open cone $\Gamma \subset \R^n$. The self-expander equation \eqref{eq:SelfSimilarInverse} to the inverse $\rho$-flow is reduced to an ODE:
\begin{equation}
\label{eq:ODE1}
\rho\left(1,\cdots,1,-\frac{r\ddot{r}}{1+\dot{r}^2}\right) = \frac{1}{\mu} \frac{r(1+\dot{r}^2)}{r-\dot{r}h}
\end{equation}
For self-expanders with $C^2$-asymptotically cylindrical ends, it is necessary that the constant $\mu$ matches with that of round cylinders. Therefore, we consider only the case where $\mu = \rho(1,\cdots,1,0)^{-1} > 0$, so that \eqref{eq:ODE1} can be written as:
\begin{equation}
\label{eq:ODE2}
\rho\left(1,\cdots,1,-\frac{r\ddot{r}}{1+\dot{r}^2}\right) = \rho(1,\cdots,1,0) (1 - \dot{r}Q)
\end{equation}
where $Q$ is defined by:
\begin{equation}
\label{eq:Q}
Q(h) := -\frac{r\dot{r}+h}{r-\dot{r}h} = -\frac{1}{2(r-\dot{r}h)}\frac{d}{dh}(r^2 + h^2).	
\end{equation}
By direct computations, it can be easily verified that:
\begin{equation}
\label{eq:1-r'Q}
1 - \dot{r}Q = \frac{r(1+\dot{r}^2)}{r-\dot{r}h}.
\end{equation}

Next we consider the quantity $\rho(1,\cdots,1,\xi)$ with an arbitrary $\xi$. Since $(1,\cdots,1,0) \in \Gamma$, we can only consider those $(\rho, \Gamma)$'s so that $(1,\cdots,1,\xi) \in \Gamma$ for any $\xi \in (-\alpha,\beta)$, where $(-\alpha,\beta)$ is an open interval containing $0$. For simplicity, define a function $\hat{\rho} : (-\alpha, \beta) \to (0, \gamma)$ given by
\[\hat{\rho}(\xi) := \rho(1,\cdots,1,\xi),\]
where $\gamma := \hat{\rho}_{\max}$.

In our construction of self-expanders, we will assume $\rho$ satisfies a uniform parabolic condition, in a sense that:
\[C \geq \pf{\hat\rho}{\xi} > 0 \quad \text{ on } (-\alpha,\beta).\]
We will further make a natural assumption that $\hat{\rho} \to 0$ as $\xi \to -\alpha$, and that $\beta > 1$ (for a technical reason). This assumption holds true if, for instance, $\rho = 0$ on $\partial\Gamma$ (then $\beta = \infty$ in this case).

Examples of $\rho$'s which fulfill these assumptions (including uniform parabolicity) include:
\begin{itemize}
\item $\rho = H$ on $\Gamma = \big\{\sum_i \lambda_i > 0\big\}$: i.e. the inverse mean curvature flow, where $\hat{\rho}(\xi) = (n-1) + \xi$ and $\pf{\hat{\rho}}{\xi} = 1$, so that $\alpha = n-1$ and $\beta = \infty$.
\item $\rho = \sigma_{k+1}/\sigma_k$ where $k < n-1$ on $\Gamma = \big\{\sigma_{k+1} > 0\big\}$: i.e. the inverse $Q_k$-flow, where
\[\hat{\rho}(\xi) = \frac{C^{n-1}_{k+1} + C^{n-1}_k \xi}{C^{n-1}_k + C^{n-1}_{k-1} \xi}, \quad \pf{\hat{\rho}}{\xi} = \frac{C^{n-1}_k C^{n-1}_k - C^{n-1}_{k-1}C^{n-1}_{k+1}}{\left(C^{n-1}_k + C^{n-1}_{k-1} \xi\right)^2},\]
so that we can pick $\alpha = C^{n-1}_{k+1} / C^{n-1}_k$, and $\beta = \infty$.
\item $\rho = \sigma_k/\sigma_{k-1} + \sum a_{ij}(\sigma_j/\sigma_i)^{1/j-i} + \sum_l b_l\sigma_l^{1/l}$ where $1 \leq i < j < k < n$ and $1 \leq l < k$, here $a_{ij}$ and $b_l$ are non-negative constants. By inspection of the graph of $\hat{\rho}$, we can take $-\alpha$ to be the root of $\hat{\rho}$ between where $\sigma_k(1,\cdots,1,\xi) = 0$ and where $\sigma_{k-1}(1,\cdots,1,\xi) \to 0$. It is uniformly parabolic since at $(1,\cdots,1,-\alpha)$, none of the $\sigma_{k-1}$, $\sigma_j$, $\sigma_i$, $\sigma_l$ goes to $0$.
\end{itemize}

The parabolic condition $\dfrac{\partial\rho}{\partial\lambda_n} > 0$ on $\Gamma$ implies that $\hat{\rho}$ is strictly increasing, and guarantee an inverse function $\hat{\rho}^{-1} : (0, \gamma) \to (-\alpha, \beta)$. Provided that 
\[\left(1,\cdots,1,-\dfrac{r\ddot{r}}{1+\dot{r}^2}\right) \in \Gamma\]
and $(1-\dot{r}Q)\hat{\rho}(0)$ is in the range of $\hat{\rho}$, i.e. $(0, \gamma)$, the ODE \eqref{eq:ODE2} is equivalent to:
\begin{equation}
\label{eq:ODE}
\ddot{r} = -\frac{(1+\dot{r}^2)}{r}\hat{\rho}^{-1}\bigg((1-\dot{r}Q)\hat{\rho}(0)\bigg)
\end{equation}

We will show that when the inverse $\rho$-flow is uniformly parabolic, then \eqref{eq:ODE} has a globally defined solution $r(h)$ on $h \in (-\infty,\infty)$ which geometrically resembles the shape of an \emph{infinite bottle}. Let us state our theorem precisely:

\begin{theorem}
\label{thm:InfiniteBottle}
Consider the inverse $\rho$-flow \eqref{eq:FlowInverse} where $\rho$ is a degree $1$ homogeneous symmetric $C^1$ function of principal curvatures $\rho(\lambda_1, \cdots, \lambda_n) : \Gamma \to \R$ such that:
\begin{itemize}
\item $(1, \cdots, 1, \xi) \in \Gamma$ for any $\xi \in (-\alpha, \beta)$, where $\alpha > 0$ and $\beta > 1$; and
\item $\rho(1,\cdots,1,\xi) \to 0$ as $\xi \to -\alpha$; and
\item there exists a constant $C > 0$ such that $C \geq \frac{\partial\rho}{\partial\lambda_n}(1,\cdots,1,\xi) > 0$ for any $\xi \in (-\alpha,\beta)$.
\end{itemize}
 Then, there exists a complete two-ended rotationally symmetric self-expander $\Sigma$ to the inverse $\rho$-flow, which is bounded between two co-axial cylinders of different radii, and $C^2$-asymptotic to these two cylinders at both ends.
\end{theorem}

\subsection{Outline of the proof}
We will establish the proof of Theorem \ref{thm:InfiniteBottle} in several steps. Here we describe our strategy, and compare our approach with Drugan, Lee, and Wheeler's construction of infinite-bottle in \cite{DLW} for the inverse mean curvature flow (the case $\rho = H$).

To construct such a self-expander, we will show the existence of a global solution $r(h) > 0$ that is strictly increasing as $h$ increases, and that $r(h)$ has finite and non-zero limit as $h \to \pm\infty$. It is intuitively clear that in order for $r(h)$ to behave in such a way, the existence of an inflection point $h_*$ at which $\ddot{r}(h_*) = 0$ is crucial. We will argue that by a careful choice of the initial condition, there is always a unique inflection point. This step is achieved in a similar fashion as in \cite{DLW} for the inverse mean curvature flow.

With the existence of a unique inflection point $h_*$, we can study the behaviors of the solution on intervals $(h_{\min}, h_*)$ and $(h_*, h_{\max})$ individually. We will show on either interval, the solution cannot extinct in finite $h$. To achieve this, there are new technical issues that were not present in \cite{DLW} where $\rho = H$. Namely, we need to ensure that $(1-\dot{r}Q)\hat{\rho}(0)$ stays within the range of $\hat{\rho}$, so that the existence theorem of ODE remains valid, and the solution obtained for \eqref{eq:ODE} is equivalent to the original self-expander ODE \eqref{eq:ODE1}. When $\rho = H$, the range of $\hat{\rho}$ is $(0, \infty)$. However, for many function $\rho$'s under our consideration such as $\rho = \frac{\sigma_{k+1}}{\sigma_k}$, the value of $\hat\rho$ is bounded from above. In this work, we will supply a new argument to rule out the possibility of $(1-\dot{r}Q)\hat{\rho}(0)$ approaching the end-points of the range of $\hat{\rho}$.

To show that $r(h)$ has a positive lower bound and a finite upper bound, the authors in \cite{DLW} uses the inverse function theorem to express $h$ as a function of $r$, and study the maximal existence interval $(r_{\min}, r_{\max})$ for $h(r)$. We will instead show $r$ is bounded directly.

\subsection{Initial conditions}
From now on, we will fix the following ``initial'' conditions\footnote{We put \emph{initial} in quote because we study the ODE solution defined on both $(h_{\min}, h_0]$ and $[h_0, h_{\max})$. From now on we will call them ``conditions'' only.} for the ODE \eqref{eq:ODE}:
\begin{align}
\label{eq:Initial_h} h_0 & < 0\\
\label{eq:Initial_r} r(h_0) & = r_0 > -h_0\\
\label{eq:Initial_r'} \dot{r}(h_0) & = \dot{r}_0 \in (0, -h_0 r_0^{-1})
\end{align}
The geometric meaning of \eqref{eq:Initial_h} and \eqref{eq:Initial_r} is that the solution starts at the region bounded between the $r$-axis and the line $r + h = 0$ in the $(r,h)$-plane (see the yellow region in Figure \ref{fig:Initial}). The condition \eqref{eq:Initial_r'} for $\dot{r}$ is implies
\[\left.\frac{d}{dh}\right|_{h=h_0}(r^2 + h^2) = 2(r_0\dot{r}_0 + h_0) < 0,\]
meaning that at $(r_0, h_0)$ the solution $r(h)$ goes into the ball $r^2 + h^2 \leq r_0^2 + h_0^2$ as $h$ increases (see the red segment in Figure \ref{fig:Initial}).

\begin{figure}[h!]
\centering
\begin{tikzpicture}
	\filldraw[yellow!50] (0,0) -- (0,3) -- (-3,3) -- (0,0);
	\draw[dashed] (0,0) circle (2);
	\draw[dashed] (3,-3) -- (-3,3);
	\draw[thick,-latex] (-3,0) -- (3,0) node[anchor=west]{$h$};
	\draw[thick,-latex] (0,-3) -- (0,3) node[anchor=south]{$r$};
	\filldraw (-1,0) circle (1pt);
	\draw[dashed] (-1,1.732) -- (-1,0) node[anchor=north]{$h_0$};
	\filldraw (-1, 1.732) circle (1pt);
	\draw[dashed] (-1, 1.732) -- (0, 1.732) node[anchor=west]{$r_0$};
	\draw[thick,red,smooth] (-1-0.3, 1.732-0.05) -- (-1,1.732) -- (-1+0.3,1.732+0.08);
	
\end{tikzpicture}
\caption{Initial conditions: $r_0 + h_0 > 0$, and $\left.\frac{d}{dh}\right|_{h_0}(r^2 + h^2) < 0$.}
\label{fig:Initial}
\end{figure}
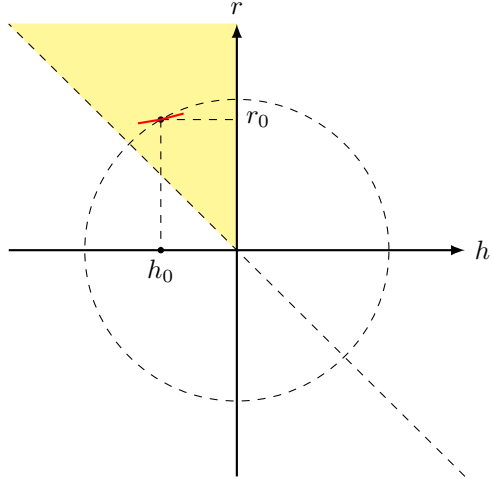

Assuming the conditions \eqref{eq:Initial_h}-\eqref{eq:Initial_r'}, the quotient $Q$ defined in \eqref{eq:Q} is initially positive at $h_0$. As such, $(1-\dot{r}Q)\hat{\rho}(0) < \hat{\rho}(0) < \gamma$ at $h = h_0$. Furthermore, by \eqref{eq:1-r'Q} we know $1-\dot{r}Q > 0$ at $h = h_0$. Therefore, the term $(1-\dot{r}Q)\hat{\rho}(0)$ lies in the range of $\hat{\rho}$ and so the term
\[\hat{\rho}^{-1}\bigg((1-\dot{r}Q)\hat{\rho}(0)\bigg)\]
in the ODE \eqref{eq:ODE} is well-defined. Thus, the existence theorem of ODE (see e.g. \cite[Theorem I.3.1]{CL55}) shows there exists a unique solution defined near $h = h_0$ which satisfies conditions \eqref{eq:Initial_h}-\eqref{eq:Initial_r'}. From now on, we denote $(h_{\min}, h_{\max})$ to be the maximal existence interval of the ODE \eqref{eq:ODE} with conditions \eqref{eq:Initial_h}-\eqref{eq:Initial_r'}.

\subsection{Monotonicity and convexity}
It is easy to observe that all constant solutions $r(h) \equiv c > 0$ satisfy the ODE \eqref{eq:ODE}. These constant solutions give round cylinders after rotating about the $h$-axis. Note that the ODE \eqref{eq:ODE} is of second-order, so the uniqueness theorem of ODE (see e.g. \cite[Theorem I.3.1]{CL55}) asserts that if two solutions to the ODE passing through the same point with the same derivative at that point, then the two solutions must be identically the same. Using this observation, we can claim that whenever a solution $r(h)$ achieves an interior critical point at $\bar{h} \in (h_{\min}, h_{\max})$ so that $\dot{r}(\bar{h}) = 0$, it is necessary that $r(h)$ is a constant solution. As such, in order to construct self-expander solution which is distinct from any round cylinder, it is necessary that the solution cannot have any critical point (i.e. either strictly increasing or strictly decreasing). By the condition \eqref{eq:Initial_r'}, we have $\dot{r}_0 > 0$ and so it is determined that $\dot{r} > 0$ on the maximal interval $(h_{\min}, h_{\max})$.

Concerning the second derivative $\ddot{r}$, it is useful to observe that the quotient $Q$ defined in \eqref{eq:Q} always have the same sign as $\ddot{r}$, by the parabolic condition of $\rho$:

\begin{lemma}
\label{lma:Qr''}
Suppose $r$ is a solution of \eqref{eq:ODE} with conditions \eqref{eq:Initial_h}-\eqref{eq:Initial_r'}. Then at any $h \in (h_{\min}, h_{\max})$, we have $Q(h) > 0$ if and only if $\ddot{r}(h) > 0$. Also, $Q(h) = 0$ if and only if $\ddot{r}(h) = 0$.
\end{lemma}

\begin{proof}
Recall that $\dot{r} > 0$. If $Q(h) > 0$, we have $[1 - \dot{r}(h)Q(h)]\hat{\rho}(0) < \hat{\rho}(0)$. Both $\hat{\rho}$ and its inverse $\hat{\rho}^{-1}$ are strictly increasing, so by the ODE \eqref{eq:ODE}, we have:
\[\ddot{r}(h) > -\frac{1+\dot{r}^2(h)}{r(h)} \hat{\rho}^{-1}\bigg(\hat{\rho}(0)\bigg) = 0.\]
Conversely, if $\ddot{r}(h) > 0$, then the ODE \eqref{eq:ODE} shows
\[\hat{\rho}^{-1}\bigg([1-\dot{r}(h)Q(h)]\hat{\rho}(0)\bigg) < 0,\]
which implies $[1-\dot{r}(h)Q(h)]\hat{\rho}(0) < \hat{\rho}(0)$. As $(1,\cdots,1,0) \in \Gamma$ and $\dot{r} > 0$, we have $Q(h) > 0$. Similarly, we can show in the same way that $Q(h) = 0$ if and only if $\ddot{r}(h) = 0$.
\end{proof}

Since the quotient $Q$ can be used to detect the convexity or concavity of $r$, it is helpful to investigate its derivative. By direct computations, we can show:
\begin{align}
\label{eq:Q'}
\dot{Q} & =	-\frac{(r-\dot{r}h)(\dot{r}^2 + r\ddot{r} + 1) + (r\dot{r}+h) \ddot{r} h}{(r-\dot{r}h)^2}\\
\nonumber & = -\frac{1+\dot{r}^2}{r-\dot{r}h} \left(1 + \frac{r\ddot{r}}{1+\dot{r}^2}\right) + \ddot{r}Q \cdot \frac{h}{r - \dot{r}h}.
\end{align}

\subsection{Behavior on $(h_{\min}, h_0]$}
We will now begin the proof of Theorem \ref{thm:InfiniteBottle}. First we discuss the behavior of the solution $r(h)$ when $h \leq h_0$.
\begin{lemma}
\label{lemma:noinflectionpoint}
There is no inflection point of $r$ on $(h_{\min}, h_0]$.
\end{lemma}

\begin{proof}
We argue by contradiction. Recall that $Q(h_0) > 0$ by \eqref{eq:Initial_h}-\eqref{eq:Initial_r'}, and so $\ddot{r}(h_0) > 0$ by Lemma \ref{lma:Qr''}. If there exists $h_*\in (h_{\min}, h_0]$ such that $\ddot{r}>0$ on $(h_*, h_0]$ and $\ddot{r}(h_*)=0$, then by  checking the signs of terms in \eqref{eq:Q'}, we have
\[\dot{Q} =	-\frac{1+\dot{r}^2}{r-\dot{r}h} \left(1 + \frac{r\ddot{r}}{1+\dot{r}^2}\right) + \ddot{r}Q \cdot \frac{h}{r - \dot{r}h} < 0\]
on $(h_*, h_0]$, which implies $Q(h_*)\geq Q(h_0)>0$. However, the hypothesis $\ddot{r}(h_*)=0$ implies $Q(h_*)=0$ again by Lemma \ref{lma:Qr''}, leading to a contradiction. Therefore, we have $\ddot{r}(h)>0$ and $Q(h)>0$ on $(h_{\min}, h_0]$.
\end{proof}

\begin{lemma}
\label{lemma:lowerbound}
There exists $c > 0$ such that $r(h) \geq c$ on $(h_{\min}, h_0]$.	
\end{lemma}

\begin{proof}
By rearranging \eqref{eq:ODE}, we can get:
\begin{equation*}
\hat{\rho}\left(-\dfrac{r\ddot{r}}{1+\dot{r}^2}\right)- \hat{\rho}(0)= -\hat{\rho}(0)\dot{r}Q.
\end{equation*}
We apply mean value theorem on the LHS, so there exists $\varsigma \in (-\frac{r\ddot{r}}{1+\dot{r}^2}, 0)$ such that
\begin{equation}
\label{eq:MVT1}
\dfrac{\partial \hat{\rho}}{\partial \xi}(\varsigma)\left(-\dfrac{r\ddot{r}}{1+\dot{r}^2}\right) =\hat{\rho}\left(-\dfrac{r\ddot{r}}{1+\dot{r}^2}\right)- \hat{\rho}(0) = -\hat{\rho}(0)\dot{r}Q.
\end{equation}
Since $\ddot{r} > 0$ on $(h_{\min}, h_0]$, we know from \eqref{eq:Q'} that $\dot{Q} < 0$ on $(h_{\min}, h_0]$, and so $Q(h) \geq Q(h_0) > 0$ on $(h_{\min}, h_0]$. Combining the uniform parabolic condition on $\rho$, we can infer from \eqref{eq:MVT1} that
\[C\dfrac{\ddot{r}}{1+\dot{r}^2}  \geq \dfrac{\partial \hat{\rho}}{\partial \xi}(\varsigma)\left(\dfrac{\ddot{r}}{1+\dot{r}^2}\right)  = \hat{\rho}(0)Q\dfrac{\dot{r}}{r} \geq \hat{\rho}(0)Q(h_0)\dfrac{\dot{r}}{r}.\]
Equivalently, we have
\begin{align*}
\dfrac{d}{dh}\left(\arctan \dot{r} - \dfrac{\hat{\rho}(0)Q(h_0)}{C} \ln r\right)\geq 0
\end{align*}
on $(h_{\min}, h_0]$, which implies
\begin{align*}
\nonumber & \arctan \dot{r} - \dfrac{\hat{\rho}(0)Q(h_0)}{C} \ln r \leq \arctan \dot{r_0} - \dfrac{\hat{\rho}(0)Q(h_0)}{C} \ln r_0 =:C_1  \\
\implies & r \geq \exp\left(\frac {-C(\frac{\pi}{2}+C_1)}{\hat{\rho}(0)Q(h_0)}\right) =: c > 0,
\end{align*}
on $(h_{\min}, h_0]$ as desired. Here we have used the fact that $\arctan \dot{r} > -\dfrac{\pi}{2}$.
\end{proof}

\begin{lemma}
\label{lemma:hmin}
$h_{\min} = -\infty$.	
\end{lemma}

\begin{proof}
According to the expression of the ODE \eqref{eq:ODE}, if $h_{\min}$ is \emph{finite}, then at least one of the following must occur:
\begin{enumerate}[(i)]
\item $r \to 0$ or $\infty$ as $h \to h_{\min}^+$.
\item $\dot{r} \to \infty$ as $h \to h_{\min}^+$.
\item $r-\dot{r}h \to 0$ or $\infty$ as $h \to h_{\min}^+$.
\item $(1-\dot{r}Q)\hat{\rho}(0) \to 0$ or $\gamma$ as $h \to h_{\min}^+$.
\end{enumerate}
We will exclude all possibilities listed above one-by-one.

By Lemma \ref{lemma:lowerbound}, we have $r \geq c > 0$ on $(h_{\min}, h_0]$, where $c$ is independent of $h_{\min}$, ruling out the possibility of $r \to 0$ as $h \to h_{\min}^+$. Since $\dot{r} > 0$ on $(h_{\min}, h_{\max})$, we have $r(h) \leq r(h_0)$ on $(h_{\min}, h_0]$. Therefore, (i) is impossible.

Next, by Lemma \ref{lemma:noinflectionpoint}, we have $\ddot{r} > 0$ on $(h_{\min}, h_0]$, which implies (ii) is impossible.

For (iii), by checking
\begin{align*}
\dfrac{d}{dh}\left(r-\dot{r}h\right) = -\ddot{r}h > 0,
\end{align*}
we obtain an upper bound for $r-\dot{r}h$ on $(h_{\min}, h_0]$:
\begin{equation*}
r-\dot{r}h \leq r_0-\dot{r}_0h_0. 
\end{equation*}
Also, we have a lower bound of $r-\dot{r}h$
\begin{equation*}
r-\dot{r}h > r \geq c > 0. 
\end{equation*}
Hence, (iii) is impossible.

Lastly, since $\dot{Q} < 0$ on $(h_{\min}, h_0]$, and so $Q(h) \geq Q(h_0) > 0$ on $(h_{\min}, h_0]$, we get:
\begin{align*}
(1-\dot{r}Q)\hat{\rho}(0) < \hat{\rho}(0)< \gamma.
\end{align*}
As $r - \dot{r}h \leq r_0-\dot{r}_0h_0$ and $r(1+\dot{r}^2) > r \geq c > 0$ on $(h_{\min}, h_0]$, then
\begin{align*}
(1-\dot{r}Q)\hat{\rho}(0) = \dfrac{r(1+\dot{r}^2)}{r-\dot{r}h}\hat{\rho}(0) \geq \dfrac{c}{r_0-\dot{r}_0h_0}\hat{\rho}(0) > 0,
\end{align*}
These show (iv) is impossible.

Therefore, $h_{\min} = -\infty$.
\end{proof}

To conclude, the solution to the ODE \eqref{eq:ODE} with conditions \eqref{eq:Initial_h}-\eqref{eq:Initial_r'} is defined on $(-\infty, h_0]$, on which it is strictly increasing, concave up, and bounded from below (see the red curve in Figure \ref{fig:InfiniteBottle}).

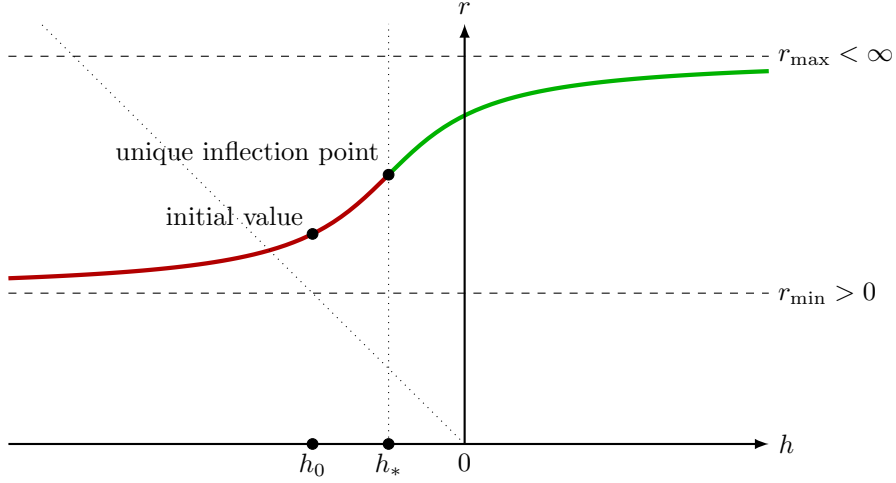
\begin{figure}[h!]
\begin{tikzpicture}
	\draw[ultra thick,domain=-5:0,smooth,variable=\x, color=black!30!red] plot({\x},{rad(atan(\x))+2});
	\draw[ultra thick,domain=0:5,smooth,variable=\x, color=black!30!green] plot({\x},{rad(atan(\x))+2});
	\draw[dashed] (-5,2-pi/2) -- (5,2-pi/2) node[anchor=west]{$r_{\min} > 0$};
	\draw[dashed] (-5,2+pi/2) -- (5,2+pi/2) node[anchor=west]{$r_{\max} < \infty$};
	\draw[thick,-latex] (-5,-pi/2) -- (5,-pi/2) node[anchor=west]{$h$};
	\draw[thick,-latex] (1,-pi/2) node[anchor=north]{$0$} -- (1,4) node[anchor=south]{$r$};
	\draw[dotted] (0,-pi/2) -- (0,4);
	\filldraw (-1,-pi/4+2) node[anchor=south east]{initial value} circle (2pt);
	\filldraw (-1,-pi/2) node[anchor=north]{$h_0$} circle (2pt);
	\filldraw (0,2) node[anchor=south east]{unique inflection point} circle (2pt);
	\filldraw (0,-pi/2) node[anchor=north]{$h_*$} circle (2pt);
	\draw[dotted] (1,-pi/2) -- (-3-pi/2,4);
\end{tikzpicture}
\caption{A plot of the profile curve}
\label{fig:InfiniteBottle}
\end{figure}

\subsection{Behavior on $[h_0, h_{\max})$}

We next find out the behavior of the solution on $[h_0, h_{\max})$. We will show there is a unique inflection point $h_*$, after which the solution becomes concave down. The solution extends globally to $[h_0, +\infty)$, and the solution is bounded from above (see the green curve in Figure \ref{fig:InfiniteBottle}).

\begin{lemma}
\label{lma:inflectionpoint}
Any inflection point $h_*$ of $r$ must be negative, and there exists a unique inflection point $h_* \in (h_0, 0)$.
\end{lemma}

\begin{proof}
For any inflection point $h_*$ of $r$, we have $\ddot{r}(h_*)=0$ and so $Q(h_*)=0$ according to Lemma \ref{lma:Qr''}. Hence, we have
\begin{align*}
\nonumber & r(h_*)\dot{r}_* + h_*=0  \\
\implies & h_* = -r(h_*)\dot{r}_* < 0.
\end{align*}

Next, we prove the existence of such $h_*$.
Suppose otherwise that $\ddot{r}(h) > 0$ (and hence $Q(h) > 0$) on $[h_0, h_{\max})$. Observe that
\begin{align*}
Q(0)= -\dfrac{r(0)\dot{r}(0)+0}{r(0)-\dot{r}(0)\cdot 0} = -\dot{r}(0) < 0.
\end{align*}
If $0 \in [h_0, h_{\max})$, it would contradict to the fact that $Q(h) > 0$ on $[h_0, h_{\max})$. Consequently, we must have $h_{\max} \leq 0$. We will show it is impossible by ruling out all possible blow-up behaviors (i)-(iv) listed below, using similar argument as in the proof of Lemma \ref{lemma:hmin}.
\begin{enumerate}[(i)]
\item $r \to 0$ or $\infty$ as $h \to h_{\max}^-$.
\item $\dot{r} \to \infty$ as $h \to h_{\max}^-$.
\item $r-\dot{r}h \to 0$ or $\infty$ as $h \to h_{\max}^-$.
\item $(1-\dot{r}Q)\hat{\rho}(0) \to 0$ or $\gamma$ as $h \to h_{\max}^-$.
\end{enumerate}
We observe that
\begin{align*}
Q>0 \implies r\dot{r}+h<0 \implies r\dot{r}{<-h\leq-h_0} \quad \mbox{ on } [h_0, h_{\max}),
\end{align*}
we can deduce on $[h_0, h_{\max})$ that:
\begin{align*}
0 < r_0\dot{r}_0 \leq r\dot{r}_0 \leq r\dot{r} < -h_0 & \implies r_0 \leq r < -\dfrac{h_0}{\dot{r}_0} \implies \text{(i) is ruled out}.
\end{align*}
\begin{align*}
0 < r_0\dot{r} \leq r\dot{r} < -h_0 & \implies \dot{r} < -\dfrac{h_0}{r_0} \implies \text{(ii) is ruled out}.
\end{align*}
\begin{align*}
r_0 \leq r \leq r-\dot{r}h \leq r-\dot{r}h_0 \leq -\dfrac{h_0}{\dot{r}_0}+\dfrac{h^2_0}{r_0} \implies \text{(iii) is ruled out}.
\end{align*}
Finally, for (iv) we consider:
\begin{align*}
\dot{r}Q > 0 \implies (1-\dot{r}Q)\hat{\rho}(0) < \hat{\rho}(0)< \gamma,
\end{align*}
and using \eqref{eq:1-r'Q}, we have
\[(1-\dot{r}Q)\hat{\rho}(0) = \frac{r(1+\dot{r}^2)}{r-\dot{r}h}\hat{\rho}(0) \geq \delta \hat{\rho}(0) > 0.\]
where $\delta := \frac{r_0}{-\frac{h_0}{\dot{r}_0}+\frac{h^2_0}{r_0}} > 0$. These rule out (iv).

To summarize, if $\ddot{r} > 0$ on $[h_0, h_{\max})$, then it is necessary that $h_{\max} \leq 0$, but the above argument shows it is not possible. Thus, there exists an inflection point $h_* \in (h_0, 0)$ such that $\ddot{r}(h_*) = 0$.

Lastly we prove that the inflection point is unique. For any inflection point $h_*$, we have $\ddot{r}(h_*) = 0$ and so by \eqref{eq:Q'}, we get:
\begin{align*}
\dot{Q}(h_*) & = -\frac{1+\dot{r}^2}{r-\dot{r}h} \bigg|_{h_*} < 0.
\end{align*}
By Lemma \ref{lma:Qr''}, we have $Q(h_*) = 0$ for any inflection point $h_*$. However, by the continuity of $Q$ on $[h_0, h_{\max})$, it is not possible to have two inflection points $h_*$ and $h_{**}$ such that $Q(h_*) = Q(h_{**}) = 0$, but $Q$ is strictly decreasing at both $h_*$ and $h_{**}$.

Therefore, the inflection point $h_*$ must be unique.
\end{proof}

\begin{lemma}
$h_{\max} = +\infty$.	
\end{lemma}

\begin{proof}
Suppose on the contrary that $h_{\max} < \infty$. Similar to the proof of $h_{\min} = -\infty$, we need to rule all possibilities below to get a contradiction:
\begin{enumerate}[(i)]
\item $r \to 0$ or $\infty$ as $h \to h_{\max}^-$.
\item $\dot{r} \to \infty$ as $h \to h_{\max}^-$.
\item $r-\dot{r}h \to 0$ or $\infty$ as $h \to h_{\max}^-$.
\item $(1-\dot{r}Q)\hat{\rho}(0) \to 0$ or $\gamma$ as $h \to h_{\max}^-$.
\end{enumerate}
To rule out (i)-(iii), one crucial observation is that $\ddot{r} < 0$ on $(h_*, h_{\max})$ by the uniqueness of the inflection point due to Lemma \ref{lma:inflectionpoint}. It follows that $\dot{r} \leq \dot{r}(h_*)$ on $[h_*, h_{\max})$. This rules out (ii). Combining with the fact that $\dot{r} > 0$, we can show on $(h_*, h_{\max})$ that:
\begin{align*}
\nonumber & r(h)-r(h_*) = \int_{h_*}^h \dot{r} \leq \dot{r}(h_*)\cdot(h - h_*) \leq \dot{r}(h_*)\cdot(h_{\max} - h_*)\\
\implies & r(h_*) \leq r(h) \leq r(h_*) + \dot{r}(h_*)\cdot(h_{\max} - h_*).
\end{align*}
These rule out (i). 

On $(h_*, 0)$, since $\dfrac{d}{dh}(r-\dot{r}h)=-\ddot{r}h < 0$, we obtain
\[r(h_*)-\dot{r}(h_*)h_* \geq r-\dot{r}h > r \geq r(h_*) > 0,\]
and on $[0, h_{\max})$, since $\dfrac{d}{dh}(r-\dot{r}h)=-\ddot{r}h\geq 0$, we obtain
\[r \geq r-\dot{r}h \geq (r-\dot{r}h)\big|_{h=0} = r(0) > 0.\]
Since $r$ has been shown to be uniformly bounded on $(h_*, h_{\max})$ whenever $h_{\max}$ is finite, the above shows $r-\dot{r}h$ is also uniformly bounded on $(h_*, h_{\max})$. This rules out (iii).

Finally we rule out (iv). The positive lower bound of $(1-\dot{r}Q)\hat{\rho}(0)$ follows from the uniform bounds derived above:

\begin{align*}
(1-\dot{r}Q)\hat{\rho}(0) = \dfrac{r(1+\dot{r}^2)}{r-\dot{r}h}\hat{\rho}(0) \geq \delta \hat{\rho}(0) >0,
\end{align*}
where $\delta = \dfrac{r(h_*)}{\max\{r(h_*)-\dot{r}(h_*)h_*, r(h_{\max})\}}$.

The more challenging part is to show that $(1-\dot{r}Q)\hat{\rho}(0)$ is bounded away from $\gamma := \hat{\rho}_{\max}$. We consider the quantity $\dot{r}Q$. By direct computations and from \eqref{eq:Q'}, we get:
\begin{equation}
\label{eq:(r'Q)'}
\frac{d}{dh}\dot{r}Q = \dot{r}\dot{Q} + \ddot{r}Q = -\dfrac{\dot{r}(1+\dot{r}^2)}{r-\dot{r}h} \left(1 + \dfrac{r\ddot{r}}{1+\dot{r}^2}\right) + \ddot{r}Q \cdot \dfrac{r}{r - \dot{r}h}.
\end{equation}

From Lemma \ref{lma:Qr''} that $Q$ and $\ddot{r}$ always have the same sign, we have $\ddot{r}Q \geq 0$. Recall that from the above discussion that we also have $r-\dot{r}h > 0$ and $r > 0$ on $[h_*, h_{\max})$, so \eqref{eq:(r'Q)'} shows
\begin{equation}
\label{eq:(r'Q)'_inequality}
\frac{d}{dh}\dot{r}Q \geq -\dfrac{\dot{r}(1+\dot{r}^2)}{r-\dot{r}h} \left(1 + \dfrac{r\ddot{r}}{1+\dot{r}^2}\right) \quad \text{ on $[h_*, h_{\max})$}.
\end{equation}

By rearranging the ODE \eqref{eq:ODE}, we get:
\begin{align*}
-\dfrac{r\ddot{r}}{1+\dot{r}^2} = \hat{\rho}^{-1}\bigg((1-\dot{r}Q)\hat{\rho}(0)\bigg).
\end{align*}
We are going to argue that the $(1-\dot{r}Q)\hat{\rho}(0)$ cannot be greater than $\hat{\rho}(1)$. Suppose there exists $h_1 \in [h_*, h_{\max})$ such that $[1-\dot{r}(h_1)Q(h_1)]\hat{\rho}(0) =  \hat{\rho}(1)$, and equivalently we have: $-\dfrac{r\ddot{r}}{1+\dot{r}^2}\bigg|_{h=h_1} = 1$. However, then \eqref{eq:(r'Q)'_inequality} implies
\[\frac{d}{dh}\bigg|_{h=h_1} \dot{r}Q > 0,\]
showing that $(1-\dot{r}Q)\hat{\rho}(0)$ is strictly decreasing at $h = h_1$. Since $(1-\dot{r}Q)\hat{\rho}(0)$ decreases whenever it reaches $\hat{\rho}(1)$. Combining with the fact that $(1-\dot{r}Q)\hat{\rho}(0) = \hat{\rho}(0) < \hat{\rho}(1)$ at $h = h_*$, we conclude that:
\begin{align*}
(1-\dot{r}Q)\hat{\rho}(0) < \hat{\rho}(1) < \gamma.
\end{align*}
Therefore, $(1-\dot{r}Q)\hat{\rho}(0)$ stays away from $0$ and $\gamma$. This rules out (iv).

Therefore, none of the blow-up behaviors (i)-(iv) can occur in $[h_*, h_{\max})$ if $h_{\max}$ were finite, we conclude that $h_{\max}=+\infty$.
\end{proof}

\begin{lemma}
\label{lma:r<K}
There exists $K < \infty$ such that $r(h) \leq K$ on $[h_0, \infty)$.
\end{lemma}

\begin{proof}
In \cite{DLW} on the inverse mean curvature flow, the authors showed $r$ is uniformly bounded from above by considering $h$ is a function of $r$, and then show $h(r)$ must blow-up when $r$ is finite. Here we will show $r$ is bounded from above directly.

Recall that $r$ is strictly increasing on $(-\infty, \infty)$, so $\displaystyle{\lim_{h\to +\infty}r}$ must exist or it is equal to $+\infty$. Suppose $\displaystyle{\lim_{h\to +\infty}r = +\infty}$, then by L'H\^opital's rule, 
\begin{align*}
\lim_{h\to +\infty}\dfrac{r}{h} = \lim_{h\to +\infty}\dfrac{\dot{r}}{1} =: \dot{r}_{\infty} \geq 0.
\end{align*}
The existence of the limit $\dot{r}_\infty$ is guaranteed by the monotonicity of $\dot{r}$ (recall that $\dot{r} > 0$ and $\ddot{r} < 0$ on $[h_*, +\infty)$). As $\frac{r}{h}$ and $\dot{r}$ approach the same finite limit as $h \to +\infty$, and $r-\dot{r}h > 0$ on $[0, \infty)$, we have $0 < \frac{r}{h} - \dot{r} < 1$ for sufficiently large $h \gg 1$. This shows:
\[Q = \dfrac{-(r\dot{r}+h)}{r-\dot{r}h} = \dfrac{-(\frac{r}{h}\dot{r}+1)}{\frac{r}{h}-\dot{r}} < -\left(\frac{r}{h}\dot{r} + 1\right) < -1\]
for sufficiently large $h \gg 1$. Combining with the fact that $Q<0$ on $[0,\infty)$ (since $\ddot{r} < 0$ on $[0,\infty)$), by compactness there exists $\varepsilon > 0$, independent of $h$, such that $Q \leq -\varepsilon < 0$ on $[0,\infty)$. Then, from \eqref{eq:ODE2}, we have
\begin{align*}
\hat{\rho}\left(-\dfrac{r\ddot{r}}{1+\dot{r}^2}\right)- \hat{\rho}(0) = -\hat{\rho}(0)\dot{r}Q \geq \varepsilon\hat{\rho}(0)\dot{r}.
\end{align*}
Applying mean value theorem to the LHS, there exists $\varsigma \in (0, -\frac{r\ddot{r}}{1+\dot{r}^2})$ such that
\begin{align*}
C\dfrac{-r\ddot{r}}{1+\dot{r}^2} & \geq \dfrac{\partial \hat{\rho}}{\partial \xi}(\varsigma)\left(\dfrac{-r\ddot{r}}{1+\dot{r}^2}\right)  = \hat{\rho}\left(-\dfrac{r\ddot{r}}{1+\dot{r}^2}\right)- \hat{\rho}(0)   \geq \varepsilon\hat{\rho}(0)\dot{r}.
\end{align*}
Then, we have
\begin{align*}
\dfrac{d}{dh}\left(\arctan \dot{r} + \dfrac{\varepsilon\hat{\rho}(0)}{C}\ln r\right) = \dfrac{\ddot{r}}{1+\dot{r}^2} + \dfrac{\varepsilon\hat{\rho}(0)}{C}\dfrac{\dot{r}}{r} \leq 0,
\end{align*}
on $[0, \infty)$ which implies
\begin{align*}
\nonumber & \arctan \dot{r} + \dfrac{\varepsilon\hat{\rho}(0)}{C}\ln r \leq \arctan \dot{r(0)} + \dfrac{\varepsilon\hat{\rho}(0)}{C}\ln r(0) =:C_2  \\
\implies & r \leq \exp\left(\frac{C(C_2+\frac{\pi}{2})}{\varepsilon\hat{\rho}(0)}\right).
\end{align*}
We have used the fact that $-\arctan \dot{r} \leq \dfrac{\pi}{2}$. However, it contradicts to our assumption that $\displaystyle{\lim_{h\to+\infty}r = +\infty}$. Therefore, we have
\[\lim_{h \to +\infty} r = K < \infty,\] 
and by $\dot{r} > 0$, we conclude that $r(h) \leq K$ on $(-\infty, \infty)$.
\end{proof}
Combining results from Lemma \ref{lemma:noinflectionpoint} to Lemma \ref{lma:r<K}, we have proved Theorem \ref{thm:InfiniteBottle}. This shows for a large class of uniformly parabolic inverse $\rho$-flows, non-compact self-expanders are not unique even with the same topological type (i.e. infinite cylinder).

\providecommand{\bysame}{\leavevmode\hbox to3em{\hrulefill}\thinspace}
\renewcommand{\MR}[1]{MR #1}
\providecommand{\MRhref}[2]{%
  \href{http://www.ams.org/mathscinet-getitem?mr=#1}{#2}
}
\providecommand{\href}[2]{#2}

\end{document}